\documentclass[12pt]{amsart}
\usepackage{amscd}
\usepackage{amsmath}
\usepackage{verbatim}

\usepackage[all, cmtip,arrow]{xy}

\usepackage{pb-diagram,pb-xy}

\usepackage{amssymb}

\numberwithin{equation}{section}

\newtheorem{thm}[equation]{Theorem}
\newtheorem{prop}[equation]{Proposition}

\newtheorem{lemma}[equation]{Lemma}
\newtheorem{cor}[equation]{Corollary}
\newtheorem{con}[equation]{Conjecture}

\theoremstyle{definition}
\newtheorem{rem}[equation]{Remark}
\newtheorem{example}[equation]{Example}
\newtheorem{dfn}[equation]{Definition}

\newcommand{\ind}{\mathop{\mathrm{ind}}}

\newcommand{\rk}{\operatorname{rk}}

\newcommand{\Ch}{\mathop{\mathrm{Ch}}\nolimits}

\newcommand{\mult}{\operatorname{mult}}

\newcommand{\Z}{\mathbb{Z}}

\newcommand{\Spec}{\operatorname{Spec}}

\newcommand{\Hom}{\operatorname{Hom}}
\newcommand{\Corr}{\operatorname{Corr}}
\newcommand{\Aut}{\operatorname{Aut}}

\renewcommand{\phi}{\varphi}

\usepackage[hypertex]{hyperref}

\title
{Decompositions of motives of generalized Severi-Brauer varieties}


\author
[M. Zhykhovich]
{Maksim Zhykhovich}

\address
{UPMC Univ Paris 06\\
Institut de Math\'ematiques de Jussieu (boite courrier 247) \\
4 place Jussieu\\
F-75252 Paris CEDEX 05\\
FRANCE}

\address
{{\it Web page:}
{\tt www.math.jussieu.fr/\~{ }zhykhovich}}

\email {zhykhovich {\it at} math.jussieu.fr}

\date
{\today}

\begin{document}

\begin{abstract}
Let $p$ be a positive prime number and $X$ be a Severi-Brauer variety of a central division algebra $D$ of degree $p^n$, with $n\geq 1$.
We describe all shifts of the motive of $X$ in the complete motivic decomposition of a variety $Y$, which splits over the function field of $X$ and satisfies the nilpotence principle. In particular, we prove the motivic decomposability of generalized Severi-Brauer varieties $X(p^m,D)$ of right ideals in $D$ of reduced dimension $p^m$, $m=0,1,\ldots,n-1$, except the cases $p=2$, $m=1$ and $m=0$ (for any prime $p$), where motivic indecomposability was proven by Nikita Karpenko.

\end{abstract}

\maketitle

\tableofcontents

\section
{Introduction}
Let $F$ be an arbitrary field and $p$ be a prime numbre. For any integer $l$, we write $v_p(l)$ for the exponent of the highest power of $p$ dividing $l$.

Let $D$ be a central division $F$-algebra of degree $p^n$, with $n\geq 1$. We
write $X(p^m,D)$ for the generalized Severi-Brauer variety of
right ideals in $D$ of reduced dimension $p^m$ for
$m=0,1,\dots,n$. In particular, $X(p^n,D)=\Spec F$ and $X(1,D)$ is the usual Severi-Brauer variety of $D$. The generalized Severi-Brauer varieties are twisted forms of grassmannians (see \cite[\S I.1.C]{MR1632779}).

For each integer $m=0,...,n$ we define an \textit{upper motive} $M_{m,D}$ in the category of Chow motives with coefficients in $\mathbb{F}_p$. This is the summand of the complete motivic decomposition of the variety $X(p^m,D)$ such that the $0$-codimensional Chow group of $M_{m,D}$ is non-zero.

Let $A$ be a central simple $F$-algebra, such that the $p$-primary component of $A$ is Brauer equivalent to $D$. Let $\mathfrak{X}_A$ be the class of finite direct products of projective $(\Aut A)$-homogeneous $F$-varieties (the class $\mathfrak{X}_A$ includes the generalized Severi-Brauer varieties of the algebra $A$).
 Nikita Karpenko proved the following theorem \cite[Theorem 3.8]{upper}. Any variety $X$ from $\mathfrak{X}_A$ decomposes into a sum of shifts of the motives $M_{m,D}$ with $m \leq v_p(\ind A_{F(X)})$. This theorem shows that the motivic indecomposable summands $M_{m,D}$ of the generalized Severi-Brauer varieties $X(p^m,D)$ are some kind of ``basic material" to construct the motives of more general class of varieties. This gives us a motivation to understand the structure of the upper motives $M_{m,D}$ themselves. It was known that in the cases $m=0$ (Severi-Brauer case, see Corollary \ref{M(SB(1,D))}) and $m=1$, $p=2$ ( \cite[Theorem 4.2]{upper}) the motive $M_{m,D}$ coincides with the whole motive of the variety $X(p^m,D)$ (that is, the motive of this variety is indecomposable). Taking into account these cases and the fact that any generalized Severi-Brauer variety $X(p^m,D)$ is $p$-incompressible \cite[Theorem 4.3]{upper} (this condition is weaker  than motivic indecomposability), one
probably expected that the Chow motive with coefficients in $\mathbb{F}_p$ of any variety
$X(p^m,D)$ is indecomposable. But, except the two already mentioned cases, the motivic decomposability of generalized Severi-Brauer variety $X(p^m,D)$ was proven in \cite{Zh}.

This article is an extended version of \cite{Zh}. To show that the motive of the variety $X(p^m,D)$ is decomposable, we prove in \cite{Zh} that some shifts of $M_{0,D}$ are the motivic summands of $X(p^m,D)$. Let $Y$ be a $F$-variety satisfying the nilpotence principle and such that it splits over the function field of $X(1,D)$. For example, one can take for $Y$ any generalized Severi-Brauer variety $X(p^m,D)$ and, more generally, any variety from $\mathfrak{X}_A$.  The main result of the present article (Theorem \ref{main}) find all shifts of $M_{0,D}$ in the complete motivic decomposition of the variety $Y$ in terms of some subgroups of rational cycles. These subgroups can be described in the case of generalized Severi-Brauer variety $X(p^m,D)$ (see Proposition \ref{Ch(XY)}). As consequence, we prove the motivic decomposability of these varieties in Corollary \ref{decomposability}.  With Theorem \ref{main} in hand, we find in
\S \ref{examples} more examples (comparing to \cite{Zh}) of complete motivic decompositions of generalized Severi-Brauer varieties $X(p^m,D)$ and therefore we describe the upper motives $M_{m,D}$ in that cases. Theorem \ref{main} also permits to prove differently (see Corollary \ref{rigidity}, \cite[Corollary 5]{Chls}) a particular case of the following conjecture.

\begin{con} \label{conj}
Let $D$ be a central division $F$-algebra. Let $K/F$ be a field extension such that $D_K$ is still division. Then $(M_{m,D})_K$ is still indecomposable.
\end{con}

\bigskip
\noindent
{\sc Acknowledgements.}
I would like to express particular
gratitude to Nikita Karpenko, my Ph.D. thesis adviser, for
introducing me to the subject, raising the question studied here,
and guiding me during this work. I am also very grateful to
Olivier Haution, Sergey Tikhonov and Skip Garibaldi for very useful discussions.

\section{Chow motives with finite coefficients}
A variety is a separated scheme of finite type over a field. Our basic reference for Chow groups and Chow motives (including
notations) is \cite{EKM}. We fix an associative unital commutative
ring $\Lambda$. Given a variety $X$ over a field $F$, we write $\mathrm{Ch}(X)$ and $\mathrm{CH}(X)$ respectively for its Chow group with
coefficients in $ \Lambda$ and for its integral Chow group.
For a field extension $L/F$ we denote by $X_L$ the respective extension of scalars.  An element of $\mathrm{Ch}(X_L)$ is called $F$-rational, if it lies in the image of the homomorphism $\mathrm{Ch}(X) \rightarrow \mathrm{Ch}(X_L)$.

Our category of motives is the category $\mathrm{CM}(F, \Lambda)$ of graded Chow motives with coefficients
in $\Lambda$, \cite[definition of \S\,64]{EKM}. By a sum of
motives we always mean the direct sum. We also write $\Lambda$ for
the motive $M(\mathrm{Spec} F) \in \mathrm{CM}(F, \Lambda)$. A
Tate motive is the motive of the form $\Lambda(i)$ with $i$ an
integer.

Let $X$ be a smooth complete variety over $F$ and let $M$ be a
motive. We call $M$ split if it is a finite sum of Tate motives.
We call $X$ split, if its integral motive $M(X)\in \mathrm{CM}(F,
\mathbb{Z})$ (and therefore the motive of $X$ with an arbitrary
coefficient ring $\Lambda$) is split. We call $M$ or $X$
geometrically split, if it splits over a field extension of $F$.
For a geometrically split variety $X$ over $F$, we denote by $\bar{X}$ the scalar extension of $X$ to a splitting field of its motive and we write
$\mathrm{\bar{Ch}}(X)$ for the subring of $F$-rational cycles in $\mathrm{Ch}(\bar{X})$. Note that the rings $\mathrm{Ch}(\bar{X})$ and $\mathrm{\bar{Ch}}(X)$ are independent on the choice of a splitting field.

Over an extension of $F$ the geometrically split motive $M$ becomes
isomorphic to a finite sum of Tate motives. We write $\mathrm{rk} \, M$ and $\mathrm{rk}_i \, M$ for
respectively the number of all summands and the number of summands
$\Lambda(i)$ in this decomposition, where $i$ is an integer. Note
that these two numbers do not depend on the choice of a
splitting field extension.

We say that $X$ satisfies the nilpotence principle, if for any
field extension $E/F$ and any coefficient ring $\Lambda$, the
kernel of the change of field homomorphism $\mathrm{End} (M(X))
\rightarrow \mathrm{End} (M(X)_E)$ consists of nilpotents. Any
projective homogeneous (under an action of a semisimple affine
algebraic group) variety is geometrically split and satisfies the
nilpotence principle, \cite[Theorem 92.4 with Remark 92.3]{EKM}.

A complete decomposition of an object in an additive category is a
finite direct sum decomposition with indecomposable summands. We
say that the Krull-Schmidt principle holds for a given object of a
given additive category, if every direct sum decomposition of the
object can be refined to a complete one (in particular, a complete
decomposition exists) and there is only one  (up to a permutation
of the summands) complete decomposition of the object. We have the
following theorem:

\begin{thm} {\upshape(\cite[Theorem 3.6 of Chapter I]{Ch-M})}.
Assume that the coefficient ring $\Lambda$ is finite. The
Krull-Schmidt principle holds for any shift of any summand of the
motive of any geometrically split $F$-variety satisfying the
nilpotence principle.
\end{thm}

We will use the following two statements in the next section.
\begin{lemma}\label{non-deg}
Assume that the coefficient ring
$\Lambda$ is a field. Let $X$ be a split variety. Then the bilinear form \mbox{ $\mathfrak{b}:\mathrm{Ch}(X)\times
\mathrm{Ch}(X) \rightarrow \Lambda$}, $\mathfrak{b}(x,y)=
\mathrm{deg}(x \cdot y)$ is non-degenerate.
\end{lemma}
\proof{
Since the motive of $X$ decomposes into a finite sum of Tate motives, we have the following decomposition for the diagonal class $\Delta \in {\Ch}_{\dim X}(X \times X)$:
$$\Delta = a_1\times b_1 + ... +a_n \times b_n \, ,$$
where $a_1, ... ,a_n$ and $b_1, ... ,b_n$ are the homogeneous elements in $\mathrm{Ch}(X)$, such that for any $i,j=1, ... ,n$ the degree $\deg(a_i \cdot b_j)\in \Lambda$ is $0$ for $i\neq j$ and $1$ for $i=j$.

Note that ${\dim}_{\Lambda} \Ch(X) = \rk M(X) = n < \infty$. Therefore, to prove the lemma it suffices to show that $\mathrm{rad} \, \mathfrak{b}= \{0\}$.
Suppose that $x \in \mathrm{rad} \, \mathfrak{b}$ (this means $\mathfrak{b}(x,y)=0$ for any $y \in \Ch(X)$). Then we have
$$\qquad\qquad\qquad\qquad x={\Delta}_*(x) = \sum^{n}_{i=1} \deg(x \cdot a_i) b_i = \sum^{n}_{i=1} \mathfrak{b}(x,a_i) b_i =0 \, . \qquad\qquad \qquad\qquad \!\! \qed$$
}

\begin{lemma} \label{lem}{Assume that the coefficient ring
$\Lambda$ is finite. Let $X$ be a variety satisfying the
nilpotence principle. Let $f \in \mathrm{End} (M(X))$ and $1_E =
f_E \in \mathrm{End}(M(X)_E)$ for some field extension $E/F$. Then
$f^n=1$ for some positive integer $n$.} \proof{Since $X$ satisfies
the nilpotence principle, we have $f=1+\varepsilon$, where
$\varepsilon$ is nilpotent. Let $n$ be a positive integer such
that ${\varepsilon}^{n}=0 = n \varepsilon$. Then
$f^{n^n}=(1+\varepsilon)^{n^n}=1$ because the binomial
coefficients $n^n \choose i$ for $i<n$ are divisible by $n$.\qed }
\end{lemma}

\section
{Main results}
\label{Main results}

Let $p$ be a positive prime integer. The coefficient ring
$\Lambda$ is $\mathbb{F}_p$ in this section. Let $F$ be a field.
Let $D$ be a central division $F$-algebra of degree $p^n$. We
write $X(p^m,D)$ for the generalized Severi-Brauer variety of
right ideals in $D$ of reduced dimension $p^m$ for
$m=0,1,\dots,n$.

\begin{lemma} \label{SB(1,D)}  Let $E/F$ be a splitting field
extension for $X=X(1,D)$. Then the subgroup of $F$-rational
cycles in $\mathrm{Ch}_{\mathrm{dim} X }(X_E \times X_E)$ is
generated by the diagonal class.
\end{lemma}
\begin{proof}
By \cite[Proposition 2.1.1]{K-i}, we have $\mathrm{\bar{Ch}}^i(X) =
0$ for $i>0$. Since the (say, first) projection $X^2 \rightarrow
X$ is a projective bundle, we have a (natural with respect to the
base field change) isomorphism $\mathrm{Ch}_{\mathrm{dim} \,
X}(X^2) \simeq \mathrm{Ch}(X)$. Passing to $\mathrm{\bar{Ch}}$, we
get an isomorphism $\mathrm{\bar{Ch}}_{\mathrm{dim} \, X}(X^2)
\simeq \mathrm{\bar{Ch}}(X) = \mathrm{\bar{Ch}}^0(X)$ showing that
$\mathrm{dim}_{\mathbb{F}_p}\mathrm{\bar{Ch}}_{\mathrm{dim} \,
X}(X^2)=1$. Since the diagonal class in
$\mathrm{\bar{Ch}}_{\mathrm{dim} \, X}(X^2)$ is non-zero, it
generates all the group.
\end{proof}

\begin{cor} \label{M(SB(1,D))} {\upshape (cf. \cite[Theorem
2.2.1]{K-i}).}{ The motive with coefficients in $\mathbb{F}_p$ of
the Severi-Brauer variety $X=X(1,D)$ is indecomposable. }
\end{cor}
\begin{proof} To prove that our motive is indecomposable it is enough to
show that $\mathrm{End}(M(X))$= $\mathrm{Ch}_{\mathrm{dim X} }(X
\times X)$ does not contain nontrivial projectors. Let $\pi \in
\mathrm{Ch}_{\mathrm{dim} X }(X \times X)$ be a projector. By
Lemma \ref{SB(1,D)}, $\pi_E$ is zero or equal to $1_E$. Since $X$
satisfies the nilpotence principle, $\pi$ is nilpotent in the
first case, but also idempotent, therefore $\pi$ is zero. Lemma
\ref{lem} gives us $\pi = 1$ in the second case.
\end{proof}

Nikita Karpenko proved the motivic indecomposability
of generalized Severi-Brauer varieties also in the case $p=2$,
$m=1$.

\begin{thm} \label{M(SB(2,D))} {\upshape (cf. \cite[Theorem 4.2]{upper})}. {Let $D$ be a central division
$F$-algebra of degree $2^n$ with $n \geq 1$. Then the motive with
coefficients in $\mathbb{F}_2$ of the variety $X(2,D)$ is
indecomposable.}
\end{thm}

Corollary \ref{decomposability} of the following main theorem will show that Corollary \ref{M(SB(1,D))} and Theorem
\ref{M(SB(2,D))} give us the only cases when the motive of generalized Severi-Brauer variety is indecomposable.

\begin{thm}
\label{main}
Let $D$ be a central division $F$-algebra of
degree $p^n$ with $n \geq 1$. Let $X$ be the Severi-Brauer variety $X(1,D)$ and $Y$ be a variety satisfying nilpotence principle, such that $Y$ is split over the function field of $X$. Then for any integer $k$ the number of copies $M(X)(k)$  in the complete motivic decomposition of $Y$ is equal to
$\dim_{\mathbb{F}_p} f_* \, \mathrm{\bar{Ch}}_{\dim Y - k}(X \times Y)$, where $f$ is a projection onto the second factor.
\end{thm}

\begin{proof}
We fix an integer $k$ and we note the motive $M(X)(k)$ simply by $M$. Let $r$ be the number of copies of $M$  in the complete motivic decomposition of $Y$. We note $V:= f_* \, \mathrm{\bar{Ch}}_{\dim Y - k}(X \times Y)$ and $r':= \dim_{\mathbb{F}_p} V$. We want to show that $r= r'$.

Let $A_1, ... , A_m$ and $B_1, ... , B_n$ be the motives. We recall that a morphism between the motives $\bigoplus_{i=1}^{m} A_i$ and $\bigoplus_{j=1}^{n} B_j$ is given by an $n \times m$-matrix of morphisms $A_i \rightarrow B_j$. The composition of morphisms is the matrix multiplication.

The motive $M^{\oplus r}$
is a summand of the motive $M(Y)$. Therefore there exist two morphisms $\alpha = (\alpha_1, ..., \alpha_r)^t \in \Hom (M^{\oplus r}, M(Y))$ and
$\beta = (\beta_1, ..., \beta_r) \in \Hom (M(Y), M^{\oplus r})$, such that $$\beta \circ \alpha = (\beta_j \circ \alpha_i)_{1\leq i, j \leq r} = (\delta_{i,j})_{1\leq i, j \leq r} \, ,$$
where $(\delta_{i,j})_{1\leq i, j \leq r}$ is the identity morphism in $\Hom(M^{\oplus r}, M^{\oplus r})$ (that is $\delta_{i,j}$ is zero if $i \neq j$ and $\delta_{i,j}$ is the diagonal class $\Delta$ in $\Corr_0(X, X)$ if $i=j$).

Let $E =F(X)$, then $E/F$ is a splitting field extension for the varieties $X$ and $Y$ (here we use the condition of the theorem) and $X_E \simeq \mathbb{P}^{d}$, where $d=p^n-1$. We know that $\Delta_E = \sum_{i=0}^{d} h^i\times h^{d-i}$, where $h$ is the hyperplane class in $\mathrm{Ch}^{1}(X_E)$.
For any $1 \leq i \leq r$ we have $$ {({\beta}_i)}_E \circ {({\alpha}_i)}_E = {({\delta}_{i,i})}_E = {\Delta}_E =h^0\times h^d + \sum_{i=1}^{d} h^i\times h^{d-i} =
[X_E]\times [pt] + \sum_{i=1}^{d} h^i\times h^{d-i}\, ,$$
where $[pt]$ is the class of a rational point in $\mathrm{Ch}(X_E)$. Therefore the correspondences ${\beta}_i \in \mathrm{Ch}_{\dim Y-k} (Y_E \times X_E)$ and ${\alpha}_i \in \mathrm{Ch}_{d+k} (X_E \times Y_E) $  have to be of the following form:
\begin{equation} \label{beta}
({\beta}_i)_E = b_i \times [pt] + \ldots \, ,
\end{equation}
where $b_i \in \mathrm{Ch}^{k}(Y_E)$ is non-zero and where ``\ldots" stands for a linear
 combination of only those terms whose first factor has
 codimension $> k$,
\begin{equation} \label{alpha}
({\alpha}_i)_E =[X_E] \times b_i^* + \ldots \, ,
\end{equation}
 where $b_i^* \in \mathrm{Ch}_{k}(Y_E)$ is such that $\mathrm{deg}(b_i
\cdot b_i^*) =1$ and where ``\ldots" stands for a linear
combination of only those terms whose second factor has
dimension $> k$.

For any $i \neq j$ we have ${({\beta}_j)}_E \circ {({\alpha}_i)}_E  = 0$, this implies that $\mathrm{deg}(b_j
\cdot b_i^*) =0$. Therefore the system of vectors $\{b^*_1, ..., b^*_r\}$ from the vector space $\Ch(Y_E)$ is dual to the system of vectors $\{b_1, ..., b_r\}$ with respect to the bilinear form $\mathfrak{b}:\mathrm{Ch}(Y_E)\times
\mathrm{Ch}(Y_E) \rightarrow \mathbb{F}_p$, \mbox{$\mathfrak{b}(x_1,x_2)=
\mathrm{deg}(x_1 \cdot x_2)$}. It follows that the vectors $b_1, ..., b_r$ are linearly independent. Since $b_i =f_*(({\beta}^t_i)_E)$, then $b_i \in V$  for any $1 \leq i \leq r$. Therefore $r \leq r'$.

Let now $b_1, ..., b_{r'}$ be a basis of $V$. We want to show that $M^{\oplus r'}$ is a motivic summand of $Y$. By the definition of $V$, there exist correspondences $\beta_1, ..., \beta_{r'} \in \Ch_{\dim Y-k}(Y \times X)$ of the form (\ref{beta}), such that $b_i = f_*(({\beta}^t_i)_E)$. Since the variety $Y_E$ is split, then by Lemma \ref{non-deg} the bilinear form $\mathfrak{b}$ is non-degenerate. It follows that there exists a system of vectors $\{b^*_1, ..., b^*_{r'}\}$ from the vector space $\Ch(Y_E)$, which is dual to the system of vectors $\{b_1, ..., b_{r'}\}$. For any $1 \leq i \leq r'$ we construct the correspondence ${\alpha}_i \in \mathrm{Ch}_{d+k} (X \times Y)$, such that ${({\alpha}_i)}_E$ is of the form (\ref{alpha}), by the following way.  The pull-back homomorphism
$$g:\mathrm{Ch}(X\times Y) \rightarrow \mathrm{Ch}(Y_{F(X)}) =
\mathrm{Ch}(Y_E)$$ with respect to the morphism
$Y_{F(X)}=(\mathrm{Spec}\, F(X))\times Y \rightarrow X \times Y$
given by the generic point of $X$ is surjective by \cite[Corollary
57.11]{EKM}. We define ${\alpha}_i \in \mathrm{Ch}(X \times Y)$ as a cycle
whose image in $\mathrm{Ch}(Y_E)$ under the surjection $g$ is $b^*_i$.
We have ${({\alpha}_i)}_E =  [X_E] \times
b^*_i + \ldots \,$, so ${({\alpha}_i)}_E$ is of the form (\ref{alpha}).

The $r'$-tuples $(\alpha_1, ..., \alpha_{r'})^t$ and $(\beta_1, ..., \beta_{r'})$ give us respectively two morphisms $\alpha \in \Hom (M^{\oplus r'}, M(Y))$ and $\beta \in \Hom (M(Y), M^{\oplus r'})$. By the construction of $\alpha$ and $\beta$, the matrix $(\mult ({(\beta_j)}_E \circ {(\alpha_i)}_E))_{1\leq i, j \leq r}$ is an identity matrix. Then, by Lemma \ref{SB(1,D)}, ${\beta}_E \circ {\alpha}_E = ({(\beta_j)_E} \circ {(\alpha_i)}_E)_{1\leq i, j \leq r} = 1_E$, where we note simply by $1$ the identity morphism $({(\delta_{i,j})})_{1\leq i, j \leq r}$ in $\Hom(M^{\oplus r}, M^{\oplus r})$. Let $\mathfrak{X}$ be a disjoint union of $r'$ copies of $X$, then $\Hom(M(\mathfrak{X}), M(\mathfrak{X}))= \Hom(M^{\oplus r'}, M^{\oplus r'})$. According to \cite[Theorem 92.4]{EKM} the variety $\mathfrak{X}$ satisfies the nilpotence principle.  By Lemma \ref{lem}, there exist a positive integer $n$, such that ${(\beta \circ \alpha)}^n = 1$ (we apply Lemma \ref{lem} to the variety $\mathfrak{X}$ and to the morphism $\beta \circ \alpha \in \Hom(M(\mathfrak{X}), M(\mathfrak{X}))$). The morphisms $\alpha$ and ${(\beta \circ \alpha)}^{n-1} \circ \beta$ give the isomorphism between the motive $M^{\oplus r'}$ and a direct summand of $M(Y)$. Therefore $r'\geq r$ and then finally $r' = r$.
\end{proof}

\begin{prop}
\label{Ch(XY)}
Let $D$ be a central division $F$-algebra of
degree $p^n$ with $n \geq 1$. Let $X$ and $Y$ be respectively the varieties $X(1,D)$ and $X(p^m,D)$, $0 \leq m <n$. Let $E/F$ be a splitting field extension for the variety $X$, let $T_1$
and $T_{p^m}$ be the tautological bundles of rank $1$ and $p^m$ on
$X_E$ and $Y_E$ respectively. Then the subring of $F$-rational cycles in $\mathrm{Ch}(X_E \times Y_E)$ is generated by the Chern classes of the vector bundle
$T_1 \boxtimes (-T_{p^m})^{\vee}$ (we lift the bundles $T_1$ and $T_{p^m}$ on $X_E \times Y_E$ and then take a product).
\end{prop}

\begin{proof}

Let $Tav$ be the tautological vector bundle on $X$. The product $X
\times Y$ considered over $X$ (via the first projection) is
isomorphic (as a scheme over $X$) to the Grassmann bundle
$G_r(Tav)$ of $r$-dimensional subspaces in $Tav$ (cf.
\cite[Proposition 4.3]{I-K}), where $r=p^n-p^m$. Let $T$ be the tautological
$r$-dimensional vector bundle on $G_r(Tav)$. By \cite[Example 14.6.6]{F}, the Chow ring $\Ch(G_r(Tav))$
as an algebra over $\Ch(X)$ is generated by Chern classes $c_0(T), c_1(T), ..., c_r(T)$.

 By \cite[Proposition 2.1.1]{K-i}, we have $\mathrm{\bar{Ch}}(X) = \mathrm{\bar{Ch}}^0(X) = \Z \cdot [X_E]$.
Therefore the Chow ring $\bar{\Ch}(X\times Y) \simeq \bar{\Ch}(G_r(Tav))$ is generated (as a ring) by Chern classes $c_0(T_E), ..., c_r(T_E)$.
Since there exists an isomorphism (cf. \cite[Proposition 4.3]{I-K}): $T_E \simeq T_1 \boxtimes
(-T_{p^m})^{\vee}$, we are done.

\end{proof}

\begin{cor}
\label{decomposability}
The motive with coefficients in
$\mathbb{F}_p$ of the variety $X(p^m,D)$ is decomposable for
$p=2$, $1 < m <n$ and for $p>2$, $0<m<n$. In these cases
$M(X(1,D))(k)$ is a summand of $M(X(p^m,D))$ for $ 2 \leq k \leq
p^n-p^m$.
\end{cor}

\begin{proof}

We use the notations: $X=X(1,D)$, $Y = X(p^m,D)$,
$d=\mathrm{dim}(X(1,D))=p^n-1$, $r=p^n-p^m$. Let $E =F(X)$, then
$E/F$ is a splitting field extension for the variety $X$ (and also
for $Y$). Over the field $E$
the algebra $D$ becomes isomorphic to $\mathrm{End}_{E}(V)$ for
some $E$-vector space $V$ of dimension $d+1=p^n$. We have $X_E
\simeq \mathbb{P}^{d}(V)$ and $Y_E \simeq G_{p^m}(V)$. Let $T_1$
and $T_{p^m}$ be the tautological bundles of rank $1$ and $p^m$ on
$X_E$ and $Y_E$ respectively. We note by $T$ the $r$-dimensional vector bundle $T_1 \boxtimes
(-T_{p^m})^{\vee}$ on $X_E \times Y_E$. By Proposition \ref{Ch(XY)}, the ring $\bar{\Ch}(X \times Y)$ is generated  by Chern classes of the vector bundle $T$.
 Let $h=c_1(T_1)\in \mathrm{Ch}^{1}(X_E)$
(then $-h$ is the hyperplane class in $\mathrm{Ch}^{1}(X_E)$) and
$c_i=c_i({(-T_{p^m})}^{\vee})\in \mathrm{Ch}^{i}(Y_E)$, $0 \leq i \leq r$. Then by
\cite[Remark 3.2.3(b)]{F}

\begin{equation}
\label{tensor}
c_t(T)=c_t(T_1
\boxtimes (-T_{p^m})^{\vee})= \sum_{i=0}^{r}(1+(h \times
1)t)^{r-i}(1\times c_i)t^i \, .
\end{equation}
It follows from the conditions
of the corollary that the binomial coefficients  ${{p^n-p^m} \choose
{2}}, {{p^n-p^m} \choose {p^m-1}}$ are divisible by $p$ and
${{p^n-p^m-1} \choose {p^m-2}} \equiv (-1)^{p^m-2}$ mod $p$.
Therefore
$$c_1(T) = (p^n-p^m)h\times 1 + 1 \times c_1 = 1 \times c_1 \,
,$$
$$c_2(T)= {{p^n-p^m} \choose {2}} h^2\times 1 + (p^n-p^m-1)h
\times c_1 + 1 \times c_2 = -h \times c_1 + 1 \times c_2 \, ,$$
\begin{multline*}
c_{p^m-1}(T)= {{p^n-p^m} \choose {p^m-1}} h^{p^m-1}\times 1 +
{{p^n-p^m-1} \choose {p^m-2}}h^{p^m-2}\times c_1 + \ldots =\\=
(-1)^{p^m-2}h^{p^m-2} \times c_1 + \ldots \, ,
\end{multline*}
 where ``\ldots"
stands for a linear
 combination of only those terms whose second factor has
 codimension $> 1$. For the top Chern class we have: $$c_r(T)=\sum_{i=0}^{r}
h^{r-i}\times c_i \, .$$

\noindent For any integer $k \geq 2$ we define $$\beta_k = c_r(T)c_{p^m-1}(T)
c_2(T) c_1(T)^{k-2} =(-h)^d \times c_1^k + \ldots \, =
 [pt] \times c_1^k + \ldots \, ,$$  where ``\ldots" stands for a linear
 combination of only those terms whose second factor has
 codimension $> k$ and where $[pt]$ is the class of a rational point in $\mathrm{Ch}(X_E)$.
Let $f: X \times Y \rightarrow X$ be a projection onto the first factor. The cycle $\beta_k$ is $F$-rational and $f_*(\beta_k) = c_1^k$.
By \cite[Example 14.6.6]{F}, the cycle $c_1^k$ is non-zero for $2 \leq k \leq p^n-p^m$. Therefore $\dim_{\mathbb{F}_p} f_* \, \mathrm{\bar{Ch}}_{\dim Y - k}(X \times Y) \geq 1$ for $2 \leq k \leq p^n-p^m$. The statement follows from Theorem \ref{main}.
\end{proof}

\begin{rem}
 The Corollary \ref{decomposability} also gives us some information
about the integral motive of the variety $X(p^m,D)$. Indeed,
according to \cite[Corollary 2.7]{P-S-Z} the decomposition of
$M(X(p^m,D))$ with coefficients in $\mathbb{F}_p$ lifts (and in a
unique way) to the coefficients $\mathbb{Z}/p^N \mathbb{Z}$ for
any $N \geq 2$. Then by \cite[Theorem 2.16]{P-S-Z} it lifts to
$\mathbb{Z}$ (uniquely for $p=2$ and $p=3$ and non-uniquely for
$p>3$). See also Remark \ref{two}.
\end{rem}

\begin{rem}
{\upshape Let $l$ be an integer such that $0<l<p^n$ and
$\mathrm{gcd}(l,p)=1$. The complete decomposition of the motive
$M(X(l,D))$ with coefficients in $\mathbb{F}_p$ is described in
\cite[Proposition 2.4]{CPSZ}}.
\end{rem}

\begin{cor}
\label{rigidity}
Let $D$ be a central division $F$-algebra of
$p$-primary index. Let $K/F$ be a field extension, such that $D_K$ is still division. Then the motive ${(M_{1,D})}_K$ is still indecomposable.
\end{cor}
\begin{proof}

We note by $X$ and $Y$ respectively the varieties $X(1,D)$ and $X(p,D)$. We note by $M$ the motive $M(X)$.
By \cite[Theorem 3.8]{upper} the complete motivic decomposition of the variety $Y$ consists of the motive $M_{1,D}$ and of the sum of motives $M$ (we neglect the shifts in this proof).  Suppose that the motive  ${(M_{1,D})}_K$ is decomposable, then by the same theorem, $M_K$ is a summand of ${(M_{1,D})}_K$. Therefore, the number of motives $M_K$ in the complete motivic decomposition of $Y_K$ is greater then the number of motives $M$ in the complete motivic decomposition of $Y$.
Let $E/K$ be a splitting field extension for the algebra $D$. By Proposition \ref{Ch(XY)}, the subspace of $K$-rational cycles in $\Ch(X_E \times Y_E)$ coincides with the subspace of $F$-rational cycles in $\Ch(X_E \times Y_E)$. Therefore the Theorem \ref{main} gives a contradiction.
\end{proof}

\section
{Complete motivic decompositions}
\label{examples}

In the Corollary \ref{decomposability} we proved that the motive of the variety $X(p^m,D)$ is decomposable for
$p=2$, $1 < m <n$ and for $p>2$, $0<m<n$. Moreover, in these cases the Corollary \ref{decomposability} gives us a list of some motivic summands of the variety $X(p^m,D)$. By duality, we can extend this list. It happens, that in two small-dimensional cases $p=3$, $m=1$, $n=2$ and $p=2$, $m=2$, $n=3$ this is already a complete list of indecomposable motivic summands of the variety $X(p^m,D)$. Note that in general it is not true (see Example \ref{example2}).

\begin{example}
\label{example1}
In this example we describe the
complete motivic decomposition of $Y:=X(3,D)$ for a division $F$-algebra $D$ of degree $9$. We note by $X$ the variety $X(1,D)$ and
by $M$ the motive $M(X)$. Note that $\dim X =8$ and $\dim Y =18$.

By \cite[Theorem 3.8]{upper}, any indecomposable motivic summand of $Y$, besides the upper motive $M_{1,D}$, is some shift of $M$. By Corollary \ref{decomposability}, the motives
$M(2)$, $M(3)$, $M(4)$, $M(5)$, $M(6)$ and by duality $M(8)$, $M(7)$ are
direct summands of $M(Y)$. Suppose that there is at least one more motive $M(t)$ (for some integer $t\geq0$) in the complete motivic decomposition of $Y$. Since by \cite[Theorem 4.3]{upper} the variety $Y$ is $3$-incompressible, we have
$$\rk_{0} M(Y) = \rk_{0} M_{1,D} = \rk_{\dim Y} M_{1,D} = \rk_{\dim Y} M(Y) = 1 \, .$$
It follows that $\rk_0 M(t) = \rk_{\dim Y} M(t) =0$. We have $1 \leq t \leq \dim Y - \dim X - 1 =9$. Since the decomposition of any of eight motives
$M(2)$, $M(3)$, ..., $M(8),M(t)$ into the sum of Tate motives over
the splitting field contains a Tate motive $\mathbb{F}_3(9)$, then $\rk_9 M_{1, D} \leq \rk_9 M(Y) - 8$. According to \cite[\S 2.5]{P-S-Poicare}, we have $\rk_9 M(Y) = 8$, therefore $\rk_9 M_{1, D} = 0$.

By \cite[Corollary 10.19]{K-Coh}, we have the following motivic decomposition of $Y$ over the function field $L=F(Y)$:
\begin{equation} \label{decff}
M(Y)_L = \bigoplus_{i+j+k =3} M(X(i, C)\times X(j, C)\times X(k,C))\, ,
\end{equation}
where $C$ is a central division $L$-algebra (of degree $3$) Brauer-equivalent to $D_L$. Note that the triples $(3,0,0)$, $(0,3,0)$, $(0,0,3)$ correspond to three Tate motives $\mathbb{F}_3$, $\mathbb{F}_3(9)$ and $\mathbb{F}_3(18)$. Let $\widetilde{M} = M(X(1,C))$, then by  \cite[Example 10.20]{K-Coh},  $M_L = \widetilde{M} \oplus \widetilde{M}(3) \oplus \widetilde{M}(6)$. It follows that the complete decomposition of $M_L$ does not contain $\mathbb{F}_3(9)$.
Therefore $\mathbb{F}_3(9)$ is a direct motivic summand of $(M_{1,D})_L$ and we have a contradiction with $\rk_9M_{1,D} =0$.

The complete motivic decomposition of the variety $X(3, D)$ with
coefficients in $\mathbb{F}_3$ is the following one:
\begin{equation} \label{decom1}
M(X(3,D))= M_{1,D} \oplus M(2) \oplus M(3) \oplus M(4) \oplus M(5)
\oplus M(6) \oplus M(7) \oplus M(8) \, .
\end{equation}
\end{example}

\begin{example} Similarly, as in the previous example, we can find the
complete motivic decomposition of $Y:=X(4,D)$ for a division $F$-algebra $D$ of degree $8$.  We note
by $M$ the motive $M(X(1,D))$.

By Corollary \ref{decomposability}, the motives
$M(2)$, $M(3)$, $M(4)$ and by duality $M(7)$, $M(6)$, $M(5)$ are
direct summands of $M(Y)$. We have
 $$M(X(4,D))=M(2)\oplus ... \oplus M(7)\oplus N$$
 \noindent for some motive $N$. Assume that $N$ is
decomposable. Then by \cite[Theorem 3.8]{upper}, and Theorems
\ref{M(SB(1,D))}, \ref{M(SB(2,D))}, the motive $N$ has an
indecomposable summand which is some shift of either $M_{0,D} = M$ or
$M_{1,D} =M(X(2,D))$. But the second case is impossible because $$70
={{8}\choose{4}}=\mathrm{rk}\, M(Y) < 6\, \mathrm{rk}M +
\mathrm{rk}\, M(X(2,D))= 6 \cdot 8 + {{8}\choose{2}}= 76 \, ,$$ 
\noindent see \cite[Example 2.18]{upper} for the computations of ranks.
Therefore $M(t)$ is a summand of $N$ for some integer $t$.

According to \cite[Corollary 10.19]{K-Coh}, we can write the
complete decomposition of $N$ over the function field
$L=F(Y)$:
$$ N_L=\mathbb{F}_2 \oplus \widetilde{M}(1) \oplus M(X(2,C))(4) \oplus M(X(2,C))(8) \oplus \widetilde{M}(12)
\oplus \mathbb{F}_2(16)\, , $$ where $C$ is a central division
$L$-algebra (of degree $4$) Brauer-equivalent to $D_L$ and where
$\widetilde{M} = M(X(1,C))$. It follows from this decomposition
that the motive $M(t)_L = \widetilde{M}(t) \oplus
\widetilde{M}(t+4)$ can not be a direct summand of $N_L$. We have a
contradiction. Therefore the motive $N$ is indecomposable and $N \simeq M_{2,D}$.\\ Now we can write
the complete motivic decomposition of $X(4,D)$ with
coefficients in $\mathbb{F}_2$:
\begin{equation} \label{decom2}
M(X(4,D))= M_{2,D} \oplus M(2) \oplus M(3) \oplus M(4) \oplus M(5)
\oplus M(6) \oplus M(7) \, .
\end{equation}
\end{example}

 Let us consider the following class of generalized Severi-Brauer varieties.
 \begin{dfn}
 We say that the generalized Severi-Brauer variety $X(p^m,D)$ is \textit{of type} 0, if the complete decomposition of $M(X(p^m,D))$ consists only of the upper motive $M_{m,D}$ and some (possibly zero) shifts of the motive $M_{0,D} = M(X(1, D))$.
 \end{dfn}
  For example, by \cite[Theorem 3.8]{upper}, the variety $X(p, D)$ is of this type. Let $Y$ be a generalized Severi-Brauer $X(p^m,D)$ variety of type $0$. By Theorem \ref{main}, the subspace of $F$-rational cycles in $\Ch(X_E \times Y_E)$ describes the complete motivic decomposition of $Y$, where $X=X(1, D)$, $E=F(X)$. Note that the structure of the ring $\Ch(X_E \times Y_E) = \Ch(X_E) \times \Ch(Y_E)$ is well-known (cf. \cite[\S\,14]{F}) and by Proposition \ref{Ch(XY)} we can compute the subring $\bar{\Ch}(X\times Y) \subset \Ch(X_E \times Y_E)$. Therefore we can say that the complete motivic decomposition of any generalized Severi-Brauer variety $X(p^m,D)$ of type $0$ can be ``theoretically'' found in a finite time using computer.

\begin{rem}
We do not possess a single example of a variety $X(p^m,D)$, which is not of type $0$. Therefore, it may happen that the generalized Severi-Brauer variety $X(p^m,D)$ is always of type $0$ (for any division $F$-algebra $D$ of degree $p^n$ and for any integer $m$, $0 \leq m \leq n$). Note that if this is true, then Conjecture \ref{conj} holds (one can follow the lines of the proof of Corollary \ref{rigidity}).
\end{rem}

\begin{example}
 \label{example2}
 Let $D$ be a central division $F$-algebra of degree $27$. In this example we find complete motivic decomposition of the variety $Y=X(3,D)$, which is of type $0$. We take the same notations as in the proof of Corollary \ref{decomposability}: $X=X(1,D)$, $E=F(X)$, $T=T_1 \boxtimes
(-T_3)^{\vee}$, where $T_1$
and $T_3$ are the tautological bundles of rank $1$ and $3$ on
$X_E$ and $Y_E$ respectively (the vector bundle $T$ is of the rank $24$). We note also by $V_{*}$ the graded $\mathbb{F}_3$-vector space  $f_* \, \mathrm{\bar{Ch}}_{\dim Y - *}(X \times Y)$, where $f$ is a projection onto the second factor.

By Theorem \ref{main}, for any integer $k$ the number of motives $M(k)$  in the complete motivic decomposition of $Y$ is equal to $\dim_{\mathbb{F}_3} V_k$, where $M=M(X)$. By duality, this number is also equal to the number of motives $M(\dim Y - \dim X - k) = M(46-k)$ in the same decomposition. Therefore the vector space $V_{\leq 23}$ describes the complete motivic decomposition of $Y$.

Let $h=c_1(T_1)\in \mathrm{Ch}^{1}(X_E)$
and $c_i=c_i({(-T_{3})}^{\vee})\in \mathrm{Ch}^{i}(Y_E)$, $0 \leq i \leq 24$. Using the formula \ref{tensor} we can compute the following Chern classes of the vector bundle $T$:
$$
\begin{array}{rl}
c_1(T)=1 \times c_1 \, , & c_2(T)=-h \times c_1 + 1\times c_2 \, ,\\
c_7(T)=1 \times c_7 \, , & c_8(T)=-h \times c_7 + 1 \times c_8 \, ,\\
\multicolumn{2}{c}{c_{24}(T)= h^{24} \times 1 + \sum_{i=1}^{24} h^{24-i}\times c_i \, .}
\end{array}
$$
We have:
$$ \begin{array}{rl}
c^2_1 = f_*\big(c_{24}(T)(c_2(T))^2\big) \in V_2 \, , & c_1c_7=f_*\big(c_{24}(T)c_2(T)c_8(T)\big)\in V_8 \, , \\
\multicolumn{2}{c}{c^2_7 = f_*\big(c_{24}(T)(c_8(T))^2 \big)\in V_{14} \, .}
\end{array}$$
Also we have the following property:
\begin{equation} \label{c1c7}
x \in V_* \Rightarrow (xc_1 \in V_{*+1} \quad \text{and} \quad xc_7 \in V_{*+7})\, .
\end{equation}
Indeed, if $x \in V_*$, then $x = f_*(y)$ for some $y \in \mathrm{\bar{Ch}}_{\dim Y - *}(X_E \times Y_E)$. Therefore $xc_1 = f_*(y\cdot c_1(T))\in V_{*+1} $ and $xc_7 = f_*(y\cdot c_7(T))\in V_{*+7}$.\\
This property gives us the following elements in $V_i$:

\begin{equation} \label{list}
\begin{array}{llll}
c^i_1 & \mbox{ if } 2 \leq i \leq 7 \, , \; & c_1^i, c_1^{i-7}c_7, c_1^{i-14}c^2_7 & \mbox{ if } 14 \leq i \leq 20 \, ,\\
c_1^i, c_1^{i-7}c_7 & \mbox{ if } 8 \leq i \leq 13 \, , \; & c_1^i, c_1^{i-7}c_7, c_1^{i-14}c^2_7, c_1^{i-21}c^3_7 & \mbox{ if } 21 \leq i \leq 23 \, .\\
\end{array}
\end{equation}
We also define a sequence $b_i$, $i \in \Z$:
$$b_i = \left\{ \begin{array}{cl}
0 & \mbox{ if } i<2 \\
1 & \mbox{ if } 2 \leq i \leq 7\\
2 & \mbox{ if } 8 \leq i \leq 13\\
3 & \mbox{ if } 14 \leq i \leq 20\\
4 & \mbox{ if } 21 \leq i \leq 23\\
b_{46-i} & \mbox{ if } i>23 .
\end{array} \right.$$
Note that for any $i \leq 23$ the number of elements lying in $V_i$ from the list \ref{list} is equal to $b_i$.


We are going to show that all elements from the list \ref{list} are linearly independent (to apply then Theorem \ref{main}).
The $\mathbb{F}_3$-vector space $V_*$ is a subspace of $\mathrm{Ch}^*(Y_E)$. We note $\tilde{c}_i =c_i(T^{\vee}_{3})$, $i=1,2,3$, where $T_{3}$ is the tautological bundle of rank $3$ on $Y_E$. According to \cite[Example 14.6.6]{F} the graded ring $\mathrm{Ch}^*(Y_E)$ is generated  by Chern classes $\tilde{c_1}, \tilde{c_2}, \tilde{c_3}, c_1,...,c_{24}$ modulo the homogeneous relations
\begin{equation} \label{relation}
c_r + c_{r-1}\tilde{c_1} + c_{r-2}\tilde{c_2}+c_{r-3}\tilde{c_3} = 0 \quad  \mbox{for} \quad r=1,..., 27 \, ,
\end{equation}
where $c_i=0$ for $i \not{\!\in} \,[0,24]$. It follows that the graded ring $\mathrm{Ch}^*(Y_E)$ is generated by only three elements $\tilde{c_1}, \tilde{c_2}, \tilde{c_3}$ modulo some homogeneous relations of degree greater than $23$. Therefore we have an isomorphism: $$\mathrm{Ch}^{* \leq 23}(Y_E) \simeq \mathbb{F}_3[\tilde{c_1}, \tilde{c_2}, \tilde{c_3}]_{\leq 23}\, .$$
 
 \noindent Using relations \ref{relation} we can compute that $c_1 = -\tilde{c_1}$ and $c_7= -\tilde{c_1}^7 + \tilde{c_1}^4\tilde{c_3} - \tilde{c_1}^3\tilde{c_2}^2 + \tilde{c_1}\tilde{c_2}^3$.

Now we consider the elements from the list \ref{list} as polynomials in $\mathbb{F}_3[\tilde{c_1}, \tilde{c_2}, \tilde{c_3}]$. To show that all of them are linearly independent, it suffices to check this for four elements $c_1^{21}, c_1^{14}c_7, c_1^{7}c^2_7, c^3_7$ (they are in our list) from $V_{21}$. Since the polynomial ring $\mathbb{F}_3[\tilde{c_1}, \tilde{c_2}, \tilde{c_3}]$ is factorial and $c_7$ is not divisible by $\tilde{c}^2_1$, then for any $\alpha, \beta, \gamma, \delta \in \mathbb{F}_3$ we have

$$\alpha c_1^{21} + \beta c_1^{14}c_7 + \gamma c_1^{7}c^2_7 + \delta c^3_7 = 0  \quad \Longrightarrow \quad \alpha=\beta=\gamma=\delta=0 \, .$$

 Since all elements from the list \ref{list} are linearly independent, then $\dim_{\mathbb{F}_3} V_i \geq b_i$ for $i \leq 23$. Therefore for any integer $i$ the motive $M^{\oplus b_i}(i)$ is a direct summand of $M(Y)$. Indeed, the statement follows from Theorem \ref{main} for $i \leq 23$ and by duality it is also true  for $i>23$. We have
\begin{equation}
M(Y)=\oplus_{i\in \Z} M^{\oplus b_i}(i) \oplus N
\end{equation}
for some motive $N$ over $F$.

Now we want to understand the complete decomposition of $N$.
Let $L$ be a function field of the variety $Y$ and $C$ be a central division $L$-algebra (of degree $9$) Brauer-equivalent to $D_L$.
Using the motivic decomposition similar to the decomposition \ref{decff} from Example \ref{example1}, we can show that the complete decomposition of $M(Y)_L$ contains three indecomposable motives: $M_{1,C}$, $M_{1,C}(27)$, $M_{1,C}(54)$. Moreover any other summand in the complete motivic decomposition of $M(Y)_L$ is a shift of the motive $\widetilde{M}: = M(X(1,C))$. We know that $M_L = \widetilde{M} \oplus \widetilde{M}(9) \oplus \widetilde{M}(18)$. It follows that  $$N_L = M_{1,C} \oplus M_{1,C}(27) \oplus M_{1,C}(54) \oplus N'$$
 \noindent for some motive $N'$ over $L$ and $N'$ is a sum of shifts of the motive $\widetilde{M}$. Note that if $M(k)$ is direct summand of $N$ for some integer $k$, then $M_L(k)$ is a direct summand of $N'$.

Let $S$ be a direct summand of the motive of a geometrically split variety. We write $P(S,t)$ for the Poincar\'e polynomial of $S$:
$$P(S,t)= \sum_{i\geq 0} \, (\rk_i S) \cdot t^i \, .$$ Let us find the Poincar\'e polynomial of the motive $N'$. We have
$$P(N',t) =  P(M(Y),t) - (1+t^{27} + t^{54})P(M_{1,C},t)- \sum_{i \geq 0} b_i t^i P(M,t) \, . $$
Using the following formulas $$P(M(Y),t) = \frac{(1-t^{27})(1-t^{26})(1-t^{25})}{(1-t)(1-t^2)(1-t^3)} \, , \mbox{ (according to \cite[\S 2.5]{P-S-Poicare})},$$
$$P(M,t)= \frac{1-t^{27}}{\! \! \! 1-t} = \sum_{i=0}^{26} t^i \, ,$$
$$ P(M_{1,C},t)=t^6 + t^{12} +\sum_{i=0}^{26} t^i   \, ,  \mbox{ (by Example \ref{decom1}) },$$
we can compute $P(N',t)$. Since $N'$ is a sum of shifts of the motive $\widetilde{M}$ then $P(N',t)$ is divisible by $P(\widetilde{M},t) = (1-t^9)/(1-t)=1+t+...+t^8$. Let $Q(t)$ be a quotient of these two polynomials. After computations, we have
\begin{multline*}
Q(t) = t^7+t^{13}+t^{16}+t^{18} + t^{19}+t^{20}+t^{22}+t^{24}+t^{26}+t^{28}+t^{29}+t^{30}+t^{34}+t^{35}+\\ t^{36}+t^{38}+
t^{40}+t^{42}+t^{44}+t^{45}+t^{46}+t^{48}+t^{51}+t^{57} \, .
\end{multline*}

Now if $M(k)$ is a direct summand of $N$ for some integer $k$, then $$M_L(k) =  \widetilde{M}(k) \oplus \widetilde{M}(k+9)\oplus \widetilde{M}(k+18)$$ is a direct summand of $N'$. Therefore in this case the decomposition of $Q(t)$ contains $t^k+t^{k+9}+t^{k+18}=P(M_L(k),t)/P(\widetilde{M},t)$. Only two values $k=20$ and $k=26$ satisfy this condition. Note that if complete decomposition of the motive $N$ contains $M(20)$ then by duality it contains also $M(26)$. It follows that the question of the complete motivic decomposition of $Y$ reduces to the question either $\dim_{\mathbb{F}_3} V_{20} = 3$ or $\dim_{\mathbb{F}_3} V_{20} = 4$?
Let us show that we are in the second case. Consider the following cycle $e$ from $V_{20}$
\begin{multline*}
 e = f_*(c^{11}_2(-T)c_3^8(-T)) = -\tilde{c}^{17}_1\tilde{c}_3 + \tilde{c}^{16}_1\tilde{c}^2_2 - \tilde{c}^{14}_1\tilde{c}^3_2 - \tilde{c}^{14}_1\tilde{c}_3^2 - \tilde{c}^{13}_1\tilde{c}_2^2\tilde{c}_3 - \tilde{c}_1^{12}\tilde{c}_2^4 + \\ \tilde{c}_1^{11}\tilde{c}_2^3\tilde{c}_3 - \tilde{c}^{11}_1\tilde{c}^3_3 - \tilde{c}^{10}_1\tilde{c}_2^5 - \tilde{c}_1^2\tilde{c}_2^9 \, ,
\end{multline*}
where $c_2(-T)= - h\tilde{c}_1 + \tilde{c}_2$, $c_3(-T)= h^3 + h^2\tilde{c}_1 + h\tilde{c}_2 +\tilde{c}_3$. The cycle $e$ as a polynomial in $\mathbb{F}_3[\tilde{c_1}, \tilde{c_2}, \tilde{c_3}]$ is not divisible by $\tilde{c}_1^3$. It follows that the cycle $e$ could not be a linear combination of three cycles $c^{20}_1,c_1^{13}c_7, c_1^{6}c^2_7$ from the list \ref{list}.
 Therefore $\dim_{\mathbb{F}_3} V_{20} = 4$.

Consider a sequence $(a_i)_{i\in \Z}$ defined by
$$a_i = \left\{ \begin{array}{ll}
b_i +1 & \mbox{ if } i=20 \mbox{ or } i=26 \\
b_i & \mbox{ else. }
\end{array} \right.$$
The complete motivic decomposition of the variety $Y$ is the following
\begin{equation} \label{decom3}
M(Y)=\oplus_{i\in \Z} M^{\oplus a_i} \oplus M_{1,D} \, .
\end{equation}
\end{example}

\begin{rem} \label{two}
We have the same decompositions as (\ref{decom1}), (\ref{decom2}) and (\ref{decom3}) for the motives with the integral coefficients. To show this one can
apply \cite[Corollary 2.7]{P-S-Z} and then \cite[Theorem
2.16]{P-S-Z}.
\end{rem}

\def\cprime{$'$}

\end{document}